\documentclass[11pt, reqno]{amsart}
\usepackage{times,amsmath,amssymb,amsthm,amscd, mathrsfs, color}
\usepackage[all]{xy}

\newtheorem{thm}{Theorem}[section]
\newtheorem{cor}[thm]{Corollary}
\newtheorem{lem}[thm]{Lemma}
\newtheorem{pro}[thm]{Proposition}

\theoremstyle{definition}
\newtheorem{df}[thm]{Definition}

\numberwithin{equation}{section}

\newcommand{\R}{{\mathbb R}}

\newcommand{\D}{{\mathbb D}}

\newcommand{\hmu}{\widehat{\mu}}

\begin{document}

\title[The weighted composition operators]
{The weighted composition operators\\ on the large weighted Bergman spaces}

\subjclass[2010]{}

\author[I. Park]{Inyoung Park}
\address{BK21-Mathematical Sciences Division, Pusan National University, Busan 46241, Republic of Korea}
\email{iypark26@gmail.com}

\thanks{}


\begin{abstract}
In this paper, we characterize bounded, compact or Schatten class weighted composition operators acting on Bergman spaces with the exponential type weights. Moreover, we give the proof of the necessary part for the boundedness of $C_\phi$ on large weighted Bergman spaces given by \cite{KM}.
\end{abstract}

\maketitle

\section{Introduction}
Let $\phi$ be an analytic self-map of the open unit disk $\D$ in the complex plane and $u$ be an analytic function on $\D$. The weighted composition operator with respect to $u$ is defined by $(uC_\phi)f(z):=u(z)f(\phi(z))$ where $f$ belongs to the holomorphic function space $H(\D)$. There was much effort to characterize those analytic maps $\phi$ which induce bounded or compact weighted composition operators on the different holomorphic function spaces (see for example, \cite{CGP,CH,CZ,OZ,PR,SU}). When $u(z)\equiv1$, it is well known that every composition operator is bounded on the standard weighted Bergman spaces by the Littlewood subordination principle.\\ \indent On the other hand, in \cite{KM}, Kriete and Maccluer showed that if there is a boundary point $\zeta$ such that the modulus of the angular derivative $|\phi'(\zeta)|$ is less than $1$ then the map $\phi$ induces an unbounded operator $C_\phi$ on the Bergman space with the fast weights. Moreover, they showed that we don't know whether $C_\phi$ is bounded or not when there is a boundary point $|\phi'(\zeta)|=1$ by giving a explicit example map.\\ \indent In this paper, we study the boundedness, the compactness and schatten class weighted composition operators $uC_\phi$ on the Bergman space with rapidly decreasing weights. For an integrable radial function $\omega$, let $L^p(\omega dA)$ be the space of all measurable functions $f$ on $\D$ such that
\begin{align*}
\|f\|^p_{p,\omega}:=\int_\D |f(z)|^p\omega(z)dA(z)<\infty,\quad0<p<\infty,
\end{align*}
where $dA(z)$ is the normalized area measure on $\D$. For a given $0<p<\infty$, the weighted Bergman space $A^{p}(\omega)$ consists of $f$ in the class of $H(\D)\cap L^p(\omega dA)$. Throughout this paper, we consider the radial weights of the form $\omega(z)=e^{-\varphi(z)}$ such that $\varphi\in\mathcal{C}^2(\D)$ and $\Delta \varphi(z)\geq C>0$ for some constant $C$. When $\varphi(z)=\alpha\log(1-|z|)$ where $\alpha>-1$, it represents the standard weighted Bergman spaces $A^p_\alpha(\D)$. In this paper, we assume that $\varphi$ satisfies the following conditions
\begin{align*}
(\Delta \varphi(z))^{-\frac{1}{2}}=:\tau(z)\searrow0,
\end{align*}
and $\tau'(z)\rightarrow0$ when $|z|\rightarrow1^-$ so, we easily check that we exclude the standard weights since $\tau(r)=1-r$ when $\varphi(r)=\log(1-r)$. In fact, these weights described above decrease faster than the standard weight $(1-|z|)^\alpha$, $\alpha>0$. We can refer to Lemma 2.3 of \cite{PP1} for the proof. Furthermore, we assume that $\tau(z)(1-|z|)^{-C}$ increases for some $C>0$ or $\lim_{|z|\rightarrow1^-}\tau'(z)\log\frac{1}{\tau(z)}=0$. Now, we say that the weight $\omega$ is in the class $\mathcal{W}$ if $\varphi$ satisfies all conditions above.
 If $\omega$ belongs to the class $\mathcal{W}$ then the associated function $\tau(z)$ has the following properties $(\mathbf{L})$:
\begin{itemize}
\item[(i)]
There exists a constant $ C_1 > 0 $ such that
\begin{align*}
\tau(z)\leq C_1(1-|z|), \quad \textrm{for}\quad z \in \D.
\end{align*}
\item[(ii)]
There exists a constant $ C_2 > 0 $ such that
\begin{align*}
|\tau(z)-\tau(\xi)|\leq C_2|z-\xi|, \quad \textrm{for}\quad z, \xi\in\D.
\end{align*}
\end{itemize}
The typical weight in the class $\mathcal{W}$ is the type of
\begin{align*}
\omega(r)=(1-r)^\beta\exp\left(\frac{-c}{(1-r)^\alpha}\right),\quad\beta\geq0,\quad\alpha,c>0,
\end{align*}
and its associated function $\tau(z)=(1-|z|)^{1+\frac{\alpha}{2}}$. We can find more examples of our weight functions and their $\tau$ functions in \cite{PP1}. Additionally, we assume that fast weights $\omega\in\mathcal{W}$ have the regularity imposing the rate of decreasing condition in \cite{KM},
\begin{align}\label{regular}
\lim_{t\rightarrow0}\frac{\omega(1-\delta t)}{\omega(1-t)}=0,\quad0<\delta<1.
\end{align}
In \cite{KM}, Kriete and Maccluer gave the following sufficient condition for the boundedness of $C_\phi$ and the equivalent conditions for the compactness of $C_\phi$ on $A^2(\omega)$.
\begin{thm}\label{knownresult}
Let $\omega$ be a regular weight in the class $\mathcal{W}$. If
\begin{align}\label{suffibddwithoutU}
\limsup_{|z|\rightarrow1}\frac{\omega(z)}{\omega(\phi(z))}<\infty,
\end{align}
then $C_\phi$ is bounded on $A^2(\omega)$. Moreover, the following conditions are equivalent:
\begin{itemize}\label{angulcompact}
\item[(1)]$C_\phi$ is compact on $A^2(\omega)$.
\item[(2)]$\lim_{|z|\rightarrow1}\frac{\omega(z)}{\omega(\phi(z))}=0$.
\item[(3)]$d_\phi(\zeta)>1$, $\forall\zeta\in\partial\D$ where $d_\phi(\zeta)=\liminf_{z\rightarrow\zeta}\frac{1-|\phi(z)|}{1-|z|}$.
\end{itemize}
\end{thm}
In the same paper, they show that (\ref{suffibddwithoutU}) is even the equivalent condition for the boundedness of $C_\phi$ on $A^2(\omega)$ when $\omega(z)=e^{-\frac{1}{(1-|z|)^\alpha}}$, $\alpha>0$. Therefore, we can expect that the condition (\ref{suffibddwithoutU}) is still the equivalent condition for the boundedness of $C_\phi$ on the Bergman spaces with weights $\omega$ in the class $\mathcal{W}$. In fact, there is almost the converse of Theorem \ref{knownresult} in \cite[Corollary 5.10]{CM} as follows:

\begin{thm}
If
\begin{align*}
\liminf_{r\rightarrow1}\frac{\omega(r)}{\omega(M_\phi(r))}=\infty
\end{align*}
then $C_\phi$ is unbounded on $A^2(\omega)$. Here, $M_\phi(r)=\sup_{\theta}|\phi(re^{i\theta})|$.
\end{thm}
In Section 3, we show that the condition (\ref{suffibddwithoutU}) is a necessary condition of the boundedness of $C_\phi$ on $A^2(\omega)$, thus we have the following result for general $p$ with the help of the Carleson measure theorem.
\begin{thm}
Let $\omega$ be a regular weight in the class $\mathcal{W}$ and $\phi$ be an analytic self-map
of $\D$. Then the composition operator $C_\phi$ is bounded on $A^p(\omega)$, $0<p<\infty$ if and only if
\begin{align*}
\limsup_{|z|\rightarrow1^-}\frac{\omega(z)}{\omega(\phi(z))}<\infty.
\end{align*}
\end{thm}
In Section 4, we studied the weighted composition operator $uC_\phi$ on $A^2(\omega)$. In \cite{CZ}, we can find the characterizations of the boundedness, compactness and membership in the Schatten classes of $uC_\phi$ defined on the standard weighted Bergman spaces in terms of the generalized Berezin transform. Using the same technique, we extended their results in \cite{CZ} to the case of large weighted Bergman spaces $A^2(\omega)$. Those our results are given by Theorem \ref{Berezin} and Theorem \ref{essentialnorm} in Section 4. Finally, we studied the membership of $uC_\phi$ in the Schatten classes $S_p$, $p>0$ in Section 5.

\bigskip

\textit{Constants.}
 In the rest of this paper, we use the notation
$X\lesssim Y$  or $Y\gtrsim X$ for nonnegative quantities $X$ and $Y$ to mean
$X\le CY$ for some inessential constant $C>0$. Similarly, we use the notation $X\approx Y$ if both $X\lesssim Y$ and $Y \lesssim X$ hold.

\section{Some preliminaries}
In this section, we recall some well-known notions and collect related facts to be used in our proofs in later section.
\subsection{Carleson type measures}
Let $\tau$ be the positive function on $\D$ satisfying the properties $\mathbf{(L)}$ introduced in Section 1. We define the Euclidean disk $D(\delta\tau(z))$ having a center at $z$ with radius $\delta\tau(z)$. Let
\begin{align}\label{carlesonbox}
m_\tau:=\frac{\min(1,C_1^{-1},C_2^{-1})}{4},
\end{align}
where $C_1, C_2$ are in the properties (i), (ii) of $\mathbf{(L)}$. In \cite{PP1}, they show that for $0<\delta<m_\tau$ and $a\in\D$,
\begin{align}\label{comparable}
\frac{1}{2}\tau(a)\leq\tau(z)\leq 2\tau(a),\quad\text{for}\quad z\in D(\delta\tau(a)).
\end{align}
The following lemma is the generalized sub-mean value theorem for $|f|^p\omega$.

\begin{lem}\cite[Lemma 2.2]{PP1}\label{subharmoni-1}
Let $\omega=e^{-\varphi}$, where $\varphi$ is a subharmonic function. Suppose the function $\tau$ satisfies properties $(\mathbf{L})$ and  $\tau(z)^2\Delta\varphi(z)\leq C$ for some constant $C>0$. For $\beta\in\R$ and $0<p<\infty$, there exists a constant $M\geq1$ such that
\begin{align*}
|f(z)|^p\omega(z)^\beta\leq\frac{M}{\delta^2\tau(z)^2}\int_{D(\delta\tau(z))}|f|^p\omega^\beta dA
\end{align*}
for a sufficiently small $\delta>0$, and $f\in H(\D)$.
\end{lem}
A positive Borel measure $\mu$ in $\D$ is called a (vanishing) Carleson measure for $A^p(\omega)$ if the embedding $A^p(\omega)\subset L^p(\omega d\mu)$ is (compact) continuous where
\begin{align}\label{embedding}
L^p(\omega d\mu):=\left\{f\in\mathcal{M}(\D)\big|\int_\D|f(z)|^p\omega(z)d\mu(z)<\infty\right\},
\end{align}
and $\mathcal{M}(\D)$ is the set of $\mu$-measurable functions on $\D$. Now, we introduce the Carleson measure theorem on $A^p(\omega)$ given by \cite{O}.

\begin{thm}[Carleson measure theorem]
Let $\omega\in\mathcal{W}$ and $\mu$ be a positive Borel measure on $\D$. Then for $0<p<\infty$, we have
\begin{itemize}
\item[(1)]The embedding $I:A^p(\omega)\rightarrow L^p(\omega d\mu)$ is bounded if and only if for a sufficiently small $\delta\in(0,m_\tau)$, we have
\begin{align*}
\sup_{z\in\D}\frac{\mu(D(\delta\tau(z)))}{\tau(z)^{2}}<\infty.
\end{align*}
\item[(2)]The embedding $I:A^p(\omega)\rightarrow L^p(\omega d\mu)$ is compact if and only if for a sufficiently small $\delta\in(0,m_\tau)$, we have
\begin{align*}
\lim_{|z|\rightarrow1}\frac{\mu(D(\delta\tau(z)))}{\tau(z)^{2}}=0.
\end{align*}
\end{itemize}
\end{thm}
From the above theorem, we note a Carleson measure is independent of $p$ so that if $\mu$ is a Carleson measure on $L^2(\omega d\mu)$ then $\mu$ is a Carleson measure on $L^p(\omega d\mu)$ for all $p$.

\subsection{The Reproducing Kernel}
Let the system of functions $\{e_k(z)\}_{k=0}^\infty$ be an orthonormal basis of $A^2(\omega)$. It is well known that the reproducing kernel for the Bergman space $A^2(\omega)$ is defined by
\begin{align*}
K(z,\xi)=\overline{K_z(\xi)}=\sum_{k=0}^\infty e_k(z)\overline{e_k(\xi)}.
\end{align*}
Unlike the standard weighted Bergman spaces, the explicit form of $K(z,\xi)$ of $A^2(\omega)$ has been unknown. However, we have the precise estimate near the diagonal given by Lemma 3.6 in \cite{LR},
\begin{align}\label{diagonalesti}
|K(z,\xi)|^2\approx K(z,z)K(\xi,\xi)\approx\frac{\omega(z)^{-1}\omega(\xi)^{-1}}{\tau(z)^2\tau(\xi)^2},\quad\xi\in D(\delta\tau(z)),
\end{align}
where $\delta\in(0,m_\tau/2)$. Recently, in \cite{AH}, they introduced the upper estimate for the reproducing kernel as follows: for $z,\xi\in\D$ there exist constants $C,\sigma>0$ such that
\begin{align}\label{kernesti}
|K(z,\xi)|\omega(z)^{1/2}\omega(\xi)^{1/2}\leq C\frac{1}{\tau(z)\tau(\xi)}\exp\left(-\sigma\inf_{\gamma}\int_0^1\frac{|\gamma'(t)|}{\tau(\gamma(t))}\,dt\right),
\end{align}
where $\gamma$ is a piecewise $\mathcal{C}^1$ curves $\gamma:[0,1]\rightarrow\D$ with $\gamma(0)=z$ and $\gamma(1)=\xi$. In the same paper, they remarked that the distance $\beta_\varphi$ denoted by
\begin{align*}
\beta_\varphi(z,\xi)=\inf_{\gamma}\int^1_0\frac{|\gamma'(t)|}{\tau(\gamma(t))}\,dt
\end{align*}
is a complete distance because of the property (i) of $(\mathbf{L})$. By the completeness of $\beta_\varphi$ and the kernel estimate (\ref{kernesti}), we obtain the following lemma.
\begin{lem}\label{testft}
Let $\omega\in\mathcal{W}$. Then the normalized kernel function $k_z(\xi)$ uniformly converges to $0$ on every compact subsets of $\D$ when $|z|\rightarrow1^-$.
\end{lem}
\begin{proof}
Given a compact subset $\mathcal{K}$ of $\D$ and $z\in\D$, (\ref{diagonalesti}) and (\ref{kernesti}) follow that
\begin{align*}
|k_z(\xi)|=\frac{|K_z(\xi)|}{\|K_z\|_{2,\omega}}\lesssim\frac{\omega(\xi)^{-1/2}}{\tau(\xi)}e^{-\sigma\beta_\varphi(z,\xi)}\leq C_{\mathcal{K}}e^{-\sigma\beta_\varphi(z,\xi)},\quad\forall\xi\in \mathcal{K}.
\end{align*}
By the completeness of $\beta_\varphi$, we conclude that $k_z(\xi)$ converges to $0$ uniformly on compact subsets of $\D$ when $z$ approaches to the boundary.
\end{proof}

\subsection{The Julia-Caratheodory Theorem}
For a boundary point $\zeta$ and $\alpha>1$, we define the nontangential approach region at $\zeta$ by
\begin{align*}
\Gamma(\zeta,\alpha)=\{z\in\D:|z-\zeta|<\alpha(1-|z|)\}.
\end{align*}
A function $f$ is said to have a nontangential limit at $\zeta$ if
\begin{align*}
\angle\lim_{\substack{z\to \zeta\\ z\in \Gamma(\zeta,\alpha)}}f(z)<\infty,\quad\text{for each}\quad\alpha>1.
\end{align*}
\begin{df}
We say $\phi$ has a finite angular derivative at a boundary point $\zeta$ if there is a point $\eta$ on the circle such that
\begin{align*}
\phi'(\zeta):=\angle\lim_{\substack{z\to \zeta\\ z\in \Gamma(\zeta,\alpha)}}\frac{\phi(z)-\eta}{z-\zeta}<\infty,\quad\text{for each}\quad\alpha>1.
\end{align*}
\end{df}
\begin{thm}[Julia-Caratheodory Theorem]\label{JC}
For $\phi:\D\rightarrow\D$ analytic and $\zeta\in\partial\D$, the following is equivalent:
\begin{itemize}
\item[(1)]$d_\phi(\zeta)=\liminf_{z\rightarrow\zeta}\frac{1-|\phi(z)|}{1-|z|}<\infty$.
\item[(2)]$\phi$ has a finite angular derivative $\phi'(\zeta)$ at $\zeta$.
\item[(3)]Both $\phi$ and $\phi'$ have (finite) nontangential limits at $\zeta$, with $|\eta|=1$ for $\eta=\lim_{r\rightarrow1}\phi(r\zeta)$.
\end{itemize}
Moreover, when these conditions hold, we have $\phi'(\zeta)=d_\phi(\zeta)\bar{\zeta}\eta$ and $d_\phi(\zeta)=\angle\lim_{z\rightarrow\zeta}\frac{1-|\phi(z)|}{1-|z|}$.
\end{thm}
In addition to the Julia-Caratheodory theorem, we use the Julia's lemma which gives a useful geometric result. For $k>0$ and $\zeta\in\partial\D$, let
\begin{align*}
E(\zeta,k)=\{z\in\D:|\zeta-z|^2\leq k(1-|z|^2)\}.
\end{align*}
A computation shows that $E(\zeta,k)=\left\{z\in\D:\left|\frac{1}{k+1}\zeta-z\right|\leq \frac{k}{1+k}\right\}$.
\begin{thm}[Julia's Lemma]\label{julialemma}
 Let $\zeta$ be a boundary point and $\phi:\D\rightarrow\D$ be analytic. If $d_\phi(\zeta)<\infty$ then $|\phi(\zeta)|=1$ where $\lim_{n\rightarrow\infty}\phi(a_n)=:\phi(\zeta)$ and $\{a_n\}$ is a sequence along which this lower limit is achieved. Moreover, $\phi(E(\zeta,k))\subseteq E(\phi(\zeta),kd_\phi(\zeta))$ for every $k>0$, that is,
\begin{align}\label{juliainequality}
\frac{|\phi(\zeta)-\phi(z)|^2}{1-|\phi(z)|^2}\leq d_\phi(\zeta)\frac{|\zeta-z|^2}{1-|z|^2}\quad\text{for all}\quad z\in\D.
\end{align}
\end{thm}
Note that (\ref{juliainequality}) shows that $d_\phi(\zeta)=0$ if and only if $\phi$ is a unimodular constant. Moreover, when $d_\phi(\zeta)\leq1$, $\phi(E(\zeta,k))\subseteq E(\phi(\zeta),k)$ thus, the set $E(\zeta,k)$ contains its image of $\phi$ when $\phi(\zeta)=\zeta$.

\subsection{Schatten Class}
For a positive compact operator $T$ on a separable Hilbert space $H$, there exist orthonormal sets $\{e_k\}$ in $H$ such that
\begin{align*}
Tx=\sum_k\lambda_k\langle x,e_k\rangle e_k,\quad x\in H,
\end{align*}
where the points $\{\lambda_k\}$ are nonnegative eigenvalues of $T$. This is referred to as the canonical form of a positive compact operator $T$. For $0<p<\infty$, a compact operator $T$ belongs to the Schatten class $S_p$ on $H$ if the sequence $\{\lambda_k\}$ belongs to the sequence space $l^p$,
\begin{align*}
\|T\|^p_{S_p}=\sum_k|\lambda_k|^p<\infty.
\end{align*}
When $1\leq p<\infty$, $S_p$ is the Banach space with the above norm and $S_p$ is a metric space when $0<p<1$. In general, if $T$ is a
compact linear operator on $H$, we say that $T\in S_p$ if $(T^*T)^{p/2}\in S_1$, $0<p<\infty$. Moreover,
\begin{align}\label{zhutext}
(T^*T)^{p/2}\in S_1 \Longleftrightarrow T^*T\in S_{p/2}.
\end{align}
In particular, when $T\in S_2$, we say that $T$ is a Hilbert-Schmidt integral operator. It is well known that every Hilbert-Schmidt operator on $L^2(X,\mu)$ is an integral operator induced by a function $K\in L^2(X\times X,\mu\times\mu)$ such that
\begin{align*}
Tf(z)=\int_XK(x,y)f(y)\,d\mu(y),\quad \forall f\in L^2(X,d\mu).
\end{align*}
The converse is also true. To study all of the basic properties of Schatten class operators above, we can refer to $\S1.4$ and $\S3.1$ in \cite{Z}.

\section{Composition Operators $C_\phi$ on $A^p(\omega)$}

In \cite{KM}, they also gave the necessary condition for the boundedness of $C_\phi$ on large Bergman spaces in terms of the angular derivative,
\begin{align}\label{necessaryangder}
d_\phi(\zeta)=\liminf_{z\rightarrow\zeta}\frac{1-|\phi(z)|}{1-|z|}\geq1,\quad\forall\zeta\in\partial\D.
\end{align}
The following Lemma enables us to consider only the boundary points $d_\phi(\zeta)=1$ in the proof of Theorem \ref{necessarycondip2}.
\begin{lem}\label{d>1}
Let $\omega$ be a regular fast weight. Then if $d_\phi(\zeta)>1$ for $\zeta\in\partial\D$ then
\begin{align*}
\lim_{z\rightarrow\zeta}\frac{\omega(z)}{\omega(\phi(z))}=0.
\end{align*}
\end{lem}
\begin{proof}
We first assume that $\liminf_{z\rightarrow\zeta}\frac{1-|\phi(z)|}{1-|z|}=d_\phi(\zeta)=1+2\epsilon<\infty$ for some $\zeta\in\partial\D$. Then for $\epsilon>0$, there exists an open disk $D_r(\zeta)$ centered at $\zeta$ with radius $r$ such that
\begin{align*}
\frac{1-|\phi(z)|}{1-|z|}>d_\phi(\zeta)-\epsilon,\quad \forall z\in D_r(\zeta)\cap\D.
\end{align*}
Thus, using the relation $x^A\geq1-A(1-x)$ when $A>1$ and $0<x<1$, we have
\begin{align}\label{maininequality}
|\phi(z)|\leq1-(d_\phi(\zeta)-\epsilon)(1-|z|)<1-(1+\epsilon)(1-|z|)\leq|z|^{1+\epsilon},
\end{align}
for $z\in D_r(\zeta)\cap\D$. On the other hand, if we take a sufficiently small $t$ such that $(1-t)^{\frac{1}{1+\epsilon}}\geq1-\frac{1+\epsilon^2}{1+\epsilon}t$, then from the assumption (\ref{regular}) we obtain the following result,
\begin{align*}
\frac{\omega(z)}{\omega(|\phi(z)|)}\leq\frac{\omega(z)}{\omega(|z|^{1+\epsilon})}=\frac{\omega((1-t)^{\frac{1}{1+\epsilon}})}{\omega(1-t)}\leq\frac{\omega(1-\frac{1+\epsilon^2}{1+\epsilon}t)}{\omega(1-t)}\longrightarrow0,
\end{align*}
since $\omega(|\phi(z)|)\geq\omega(|z|^{1+\epsilon})$ by (\ref{maininequality}). When $d_\phi(\zeta)=\liminf_{z\rightarrow\zeta}\frac{1-|\phi(z)|}{1-|z|}=\infty$, we have the prompt result, $|\phi(z)|<|z|^M$, $M>1$ for near $\zeta$ by the definition, since there exists $M>1$ such that
\begin{align*}
\log\frac{1}{|\phi(z)|}>M\log\frac{1}{|z|},
\end{align*}
for $|z|$ near $1$. Thus, we complete the proof.
\end{proof}

We now prove our first theorem.
\begin{thm}\label{necessarycondip2}
Given a regular weight $\omega$ in $\mathcal{W}$, let $\phi$ be an analytic self-map
of $\D$. If the composition operator $C_\phi$ is bounded on $A^2(\omega)$, then
\begin{align*}
\limsup_{z\rightarrow\zeta\in\partial\D}\frac{\omega(z)}{\omega(\phi(z))}<\infty.
\end{align*}
\end{thm}
\begin{proof}
Since the composition operator is bounded, we have $d_\phi(\zeta)\geq1$ for all boundary points $\zeta$ by (\ref{necessaryangder}). By Lemma \ref{d>1}, it suffices to check only the boundary point $\zeta$ satisfying $d_\phi(\zeta)=1$. For a boundary point $\zeta$ such that $d_\phi(\zeta)=1$, we have the following relation of inclusion by Theorem \ref{julialemma},
\begin{align*}
\phi\left(E\left(\zeta,\frac{1-r}{1+r}\right)\right)\subseteq E\left(\phi(\zeta),\frac{1-r}{1+r}\right),\quad r\in(0,1).
\end{align*}
Here, the set $E(\zeta,\frac{1-r}{1+r})$ introduced in Section 2.3 is the closed disk centered at $\frac{1+r}{2}\zeta$ with radius $\frac{1-r}{2}$ so, $r\zeta$ is the point on the boundary of $E(\zeta,\frac{1-r}{1+r})$ closest to $0$. Thus, we have
\begin{align*}
\left|\frac{1+r}{2}\phi(\zeta)-\phi(r\zeta)\right|\leq\frac{1-r}{2},
\end{align*}
that is, $|\phi(r\zeta)|\geq r$ for $0<r<1$ since $|\phi(\zeta)|=1$ by Theorem \ref{JC}. Therefore, we have $r_0\in(0,1)$ such that
\begin{align}\label{tau1}
\tau(\phi(|z|\zeta))\leq\tau(z)\quad\text{for}\quad r_0\leq|z|<1,
\end{align}
since $\tau(z)$ is a radial decreasing function when $|z|\rightarrow1^-$.
Now, define the following radial function
\begin{align*}
M_{\phi}(r)=\sup_{\eta\in\partial\D}|\phi(r\eta)|,\quad 0<r<1.
\end{align*}
Then (\ref{tau1}) follows that
\begin{align}\label{keyinequal}
\tau(M_{\phi}(|z|))\leq\tau(\phi(|z|\zeta))\leq\tau(z),\quad\text{for}\quad r_0<|z|<1.
\end{align}
Now, we assume that there is a sequence $\{z_n\}$ converges to $\zeta$ with $\frac{\omega(z_n)}{\omega(\phi(z_n))}\rightarrow\infty$ when $n\rightarrow\infty$. We can choose $\{\xi_n\}$ in $\D$ such that $|\xi_n|=|z_n|$ and $|\phi(\xi_n)|=M_\phi(|z_n|)$.
Using the relation $C_\phi^*K_z(\xi)=\langle C_\phi^*K_z,K_\xi\rangle=\langle K_z,C_\phi K_\xi\rangle=K_{\phi(z)}(\xi)$, we have
\begin{align*}
\|C_\phi^*k_{z}\|^2_{2,\omega}=\frac{K_{\phi(z)}(\phi(z))}{\|K_{z}\|^2_{2,\omega}}\approx\frac{\tau(z)^2}{\tau(\phi(z))^2}\frac{\omega(z)}{\omega(\phi(z))},\quad\forall z\in\D,
\end{align*}
thus by (\ref{keyinequal}), we obtain
\begin{align*}
\|C_\phi\|^2\gtrsim\|C_\phi^*k_{\xi_n}\|^2_{2,\omega}&\gtrsim\frac{\omega(\xi_n)}{\omega(\phi(\xi_n))}\frac{\tau(\xi_n)^2}{\tau(\phi(\xi_n))^2}\\&=\frac{\omega(z_n)}{\omega(M_{\phi}(|z_n|))}\frac{\tau(z_n)^2}{\tau(M_{\phi}(|z_n|))^2}
\geq\frac{\omega(z_n)}{\omega(\phi(z_n))},
\end{align*}
for large $n>N$. Thus, we have its contradiction and complete our proof.
\end{proof}
From Theorem \ref{necessarycondip2} together with Theorem \ref{knownresult}[KM], we conclude that (\ref{suffibddwithoutU}) is the equivalent condition for the boundedness of $C_\phi$ on $A^2(\omega)$. Furthermore, by the change of variables in the measure theory in Theorem C in $\S 39$ of \cite{H}, we have
\begin{align*}
\|C_\phi f\|^p_{p,\omega}&=\int_\D|f\circ\phi(z)|^p\omega(z)\,dA(z)=\int_\D|f(z)|^p[\omega dA]\circ\phi^{-1}(z).
\end{align*}
Denote
\begin{align*}
d\mu^\phi_{\omega}(z):=\omega(z)^{-1}[\omega dA]\circ\phi^{-1}(z).
\end{align*}
Then $C_\phi$ is (compact) bounded on $A^p(\omega)$ for $0<p<\infty$ if and only if the measure $\mu^\phi_\omega$ is a (vanishing) Carleson measure. Therefore the Carleson measure theorem shows that the condition (\ref{suffibddwithoutU}) is still valid for the boundedness for all ranges of $p$.

\begin{thm}\label{pextension}
Given a regular weight $\omega$ in $\mathcal{W}$, let $\phi$ be an analytic self-map
of $\D$. Then the composition operator $C_\phi$ is bounded on $A^p(\omega)$, $0<p<\infty$ if and only if
\begin{align}\label{u=1bdd}
\limsup_{|z|\rightarrow1}\frac{\omega(z)}{\omega(\phi(z))}<\infty.
\end{align}
\end{thm}

\section{Weighted Composition Operators $uC_\phi$ on $A^2(\omega)$}
For a positive Borel measure $\mu$, we introduce the Berezin transform $\widetilde{\mu}$ on $\D$ given by
\begin{align*}
\widetilde{\mu}(z)=\int_\D|k_z(\xi)|^2\omega(\xi)d\mu(\xi),\quad z\in \D,
\end{align*}
where $k_z$ is the normalized kernel of $A^2(\omega)$.
 For $\delta\in(0,1)$, the averaging function $\widehat{\mu}_\delta$ over the disks $D(\delta\tau(z))$ is defined by
\begin{align*}
\widehat{\mu}_\delta(z)=\frac{\mu(D(\delta\tau(z)))}{|D(\delta\tau(z))|},\quad z\in \D.
\end{align*}
In \cite{APP}, we already calculated the following relation between the Berezin transform and the averaging function,
\begin{align}\label{averbere}
\widehat{\mu}_\delta(z)\leq C\widetilde{\mu}(z),\quad z\in\D.
\end{align}
Moreover, we can see that (2),(3), and (4) in the following Proposition are equivalent in \cite{APP}. Before proving Proposition \ref{rescarl}, we introduce the covering lemma which plays an essential role in the proof of many theorems including Carleson measure theorem.
\begin{lem}\cite{O}\label{oleinik}
Let $\tau$ be a positive function satisfying properties $(\mathbf{L})$ and $\delta\in(0,m_\tau)$. Then there exists a sequence of points $\{a_k\}\subset\D$ such that
\begin{itemize}
\item[(1)] $a_j\notin D(\delta\tau(a_k))$, for $j\neq k$.
\item[(2)] $\bigcup_k D(\delta\tau(a_k))=\D$.
\item[(3)] $\widetilde{D}(\delta\tau(a_k))\subset D(3\delta\tau(a_k))$, where $\widetilde{D}(\delta\tau(a_k))=\bigcup_{z\in D(\delta\tau(a_k))} D(\delta\tau(z))$
\item[(4)] $\{D(3\delta\tau(a_k))\}$ is a covering of $\D$ of finite multiplicity $N$.
\end{itemize}
Here, we call the sequence $\{a_k\}_{k=1}^\infty$ $\delta$-sequence.
\end{lem}
Now, we characterize Carleson measures on $A^2(\omega)$ in terms of averaging function and Berezin transform.
\begin{pro}\label{rescarl}
Let $\omega\in\mathcal{W}$ and $\delta\in(0,m_\tau/2)$. For $\mu\geq0$, the following conditions are equivalent:
\begin{itemize}
\item[(1)]$\mu$ is a Carleson measure on $A^2(\omega)$.
\item[(2)]$\widetilde{\mu}$ is bounded on $\D$.
\item[(3)]$\hmu_{\delta}$ is bounded on $\D$.
\item[(4)] The sequence $\{\widehat{\mu}_\delta(a_k)\}$ is bounded for every $\delta$- sequence $\{a_k\}$.
\end{itemize}
\end{pro}
\begin{proof}
It is clear that (1) $\Rightarrow$ (2). Thus, we remain to show that (4) implies (1). By Lemma \ref{subharmoni-1}, (3) of Lemma \ref{oleinik} and (\ref{comparable}), we have
\begin{align*}
\sup_{\xi\in D(\delta\tau(a_k))}|f(\xi)|^2\omega(\xi)&\lesssim\sup_{\xi\in D(\delta\tau(a_k))}\frac{1}{\tau(\xi)^2}\int_{D(\delta\tau(\xi))}|f|^2\omega dA\\
&\lesssim\frac{1}{\tau(a_k)^2}\int_{D(3\delta\tau(a_k))}|f|^2\omega dA.
\end{align*}
Therefore, by Lemma \ref{oleinik} and the inequality above, we have the desired result,
\begin{align}\label{calcul}
\int_\D|f(\xi)|^2\omega(\xi)d\mu(\xi)&\leq\sum_{k=1}^\infty\int_{D(\delta\tau(a_k))}|f(\xi)|^2\omega(\xi)d\mu(\xi)\nonumber\\
&\leq\sum_{k=1}^\infty\mu(D(\delta\tau(a_k)))\sup_{\xi\in D(\delta\tau(a_k))}|f(\xi)|^2\omega(\xi)\nonumber\\
&\lesssim\sum_{k=1}^\infty\frac{\mu(D(\delta\tau(a_k)))}{|D(\delta\tau(a_k))|}\int_{D(3\delta\tau(a_k))}|f|^2\omega dA\nonumber\\
&\lesssim \sup_{a\in\D}\hmu_\delta(a) N\|f\|^2_{2,\omega}.
\end{align}
\end{proof}

For an analytic self-map $\phi$ of $\D$ and a function $u\in L^1(\D)$, we define the $\phi$-Berezin transform of $u$ by
\begin{align*}
B_\phi u(z)=\int_\D |k_z(\phi(\xi))|^2u(\xi)\omega(\xi)dA(\xi).
\end{align*}

\begin{thm}\label{Berezin}
Let $\omega\in\mathcal{W}$ and $u$ be an analytic function on $\D$. Then the weighted composition operator $uC_\phi$ is bounded on $A^2(\omega)$ if and only if $B_\phi(|u|^2)\in L^\infty(\D)$. Moreover, $\|uC_\phi\|\approx\sup_{z\in\D}B_\phi(|u|^2)(z)$.
\end{thm}
\begin{proof}
Given an analytic function $u$, we let $d\mu_{|u|^2}(z)=|u(z)|^2\omega(z)dA(z)$. Now, we define the positive measure
\begin{align}\label{Carlesonmeasure}
d\mu^\phi_{\omega,u}(z):=\omega(z)^{-1}d\mu_{|u|^2}\circ\phi^{-1}(z).
\end{align}
By the change of variables in the measure theory, we have
\begin{align*}
\int_\D|u(z)(f\circ\phi)(z)|^2\omega(z)\,dA(z)&=\int_\D|(f\circ\phi)(z)|^2\,d\mu_{|u|^2}(z)\\&=\int_\D|f(z)|^2\omega(z)\,d\mu_{\omega,u}^\phi(z)
\end{align*}
Thus, the fact that the measure $d\mu_{\omega,u}^\phi$ is a Carleson measure is equivalent to that the Berezin transform $\widetilde{\mu_{\omega,u}^\phi}$ is bounded on $\D$ by Proposition \ref{rescarl}.
On the other hand,
\begin{align}\label{Berezinphi}
\widetilde{\mu^\phi_{\omega,u}}(z)&=\int_\D|k_z(\xi)|^2\omega(\xi)\,d\mu^\phi_{\omega,u}(\xi)\nonumber\\
&=\int_\D|k_z(\phi(\xi))|^2|u(\xi)|^2\omega(\xi)\,dA(\xi)=B_\phi(|u|^2)(z),
\end{align}
for all $z\in\D$. Thus, the proof is complete. Furthermore, from (\ref{calcul}) and (\ref{averbere}), we have
\begin{align*}
\|uC_\phi\|\lesssim\sup_{z\in\D}B_\phi(|u|^2)(z).
\end{align*}
Finally, we obtain $\|uC_\phi\|\approx\sup_{z\in\D}B_\phi(|u|^2)(z)$ since $B_\phi(|u|^2)(z)\leq\|uC_\phi\|$ for all $z\in\D$.

\end{proof}

As an immediate consequence of Theorem \ref{Berezin} we obtain more useful necessary condition for the boundedness of $uC_\phi$.
\begin{cor}
Let $\omega\in\mathcal{W}$ and $u$ be an analytic function $\D$. If $uC_\phi$ is bounded on $A^2(\omega)$, then
\begin{align}\label{problem}
\sup_{z\in\D}\frac{\tau(z)}{\tau(\phi(z))}\frac{\omega(z)^{1/2}}{\omega(\phi(z))^{1/2}}|u(z)|<\infty.
\end{align}
\end{cor}
\begin{proof}
For any $z\in\D$, Theorem \ref{Berezin} follows that
\begin{align*}
\infty>B_\phi(|u|^2)(\phi(z))&=\int_\D|k_{\phi(z)}(\phi(\xi))|^2|u(\xi)|^2\omega(\xi)\,dA(\xi)\\
&\geq\int_{D(\delta\tau(z))}|k_{\phi(z)}(\phi(\xi))|^2|u(\xi)|^2\omega(\xi)\,dA(\xi)\\
&\gtrsim\tau(z)^2|k_{\phi(z)}(\phi(z))|^2|u(z)|^2\omega(z)\\
&\gtrsim\frac{\tau(z)^2}{\tau(\phi(z))^2}\frac{\omega(z)}{\omega(\phi(z))}|u(z)|^2.
\end{align*}
\end{proof}


Before characterizing the compactness of $uC_\phi$, we recall the definition of the essential norm $\|T\|_e$ for a bounded operator $T$ on a Banach space $X$ as follows:
\begin{align*}
\|T\|_e=\inf\{\|T-K\|:K \ is\ any\ compact\ operator\ on\ X\}.
\end{align*}
For a bounded operator $uC_\phi$ on $A^2(\omega)$, we use the following formula introduced in \cite{S} to estimate the essential norm of $uC_\phi$,
\begin{align}\label{equiessennorm}
\|uC_\varphi\|_e=\lim_{n\rightarrow\infty}\|uC_\varphi R_n\|,
\end{align}
where $R_n$ is the orthogonal projection of $A^2(\omega)$ onto $z^nA^2(\omega)$ defined by
\begin{align*}
R_nf(z)=\sum^\infty_{k=n}a_kz^k\quad\text{for}\quad f(z)=\sum^\infty_{k=0}a_kz^k.
\end{align*}
The formula (\ref{equiessennorm}) can be obtained by the similar proof of Proposition 5.1 in \cite{S}. In fact, we can find the proof of (\ref{equiessennorm}) adjusting to large weighted Bergman spaces in \cite[p.775]{KM}.
Now, we prove the following estimate for the essential norm of a bounded weighted composition operator $uC_\phi$ on $A^2(\omega)$.

\begin{thm}\label{essentialnorm}
Let $\omega\in\mathcal{W}$ and $u$ be an analytic function on $\D$. If $uC_\phi$ is bounded on $A^2(\omega)$, then there is an absolute constant $C\geq1$ such that
\begin{align*}
\limsup_{|z|\rightarrow1}B_\phi(|u|^2)(z)\leq\|uC_\phi\|_e\leq C\limsup_{|z|\rightarrow1}B_\phi(|u|^2)(z).
\end{align*}
\end{thm}
\begin{proof}
For $\|f\|_{2,\omega}\leq1$ and a fixed $0<r<1$, we have
\begin{align*}
\|(uC_\phi R_n)f\|_{2,\omega}^2&=\int_\D|R_nf(\xi)|^2\omega(\xi)\,d\mu^\phi_{\omega,u}(\xi)\\
&=\int_{\D\setminus r\D}|R_nf(\xi)|^2\omega(\xi)\,d\mu^\phi_{\omega,u}(\xi)+\int_{r\D}|R_nf(\xi)|^2\omega(\xi)\,d\mu^\phi_{\omega,u}(\xi),
\end{align*}
where $r\D=\{z\in\D:|z|\leq r\}$. By the orthogonality of $R_n$, we have $|R_nf(\xi)|\leq\|f\|_{2,\omega}\|R_n K_\xi\|_{2,\omega}$
and the following series is uniformly bounded on $|\xi|\leq r<1$,
\begin{align*}
\|R_n K_\xi\|^2_{2,\omega}\lesssim\sum_{k=n}^\infty\frac{1}{p_k}r^{2k}<\infty,\quad\text{where}\quad p_k=\int^1_0s^{2k+1}\omega(s)ds.
\end{align*}
Thus, $\|R_n K_\xi\|\rightarrow0$ as $n\rightarrow\infty$ so that $|R_nf(\xi)|$ also uniformly converges to $0$ on $r\D$ as $n\rightarrow\infty$. Therefore the second integral vanishes as $n\rightarrow\infty$ since $\mu^\phi_{\omega,u}$ is a Carleson measure. For the first integral, we denote by $\mu^\phi_{\omega,u,r}:=\mu^\phi_{\omega,u}|_{\D\setminus r\D}$. For a fixed $r>0$, we easily calculate $(\D\setminus r\D)\cap D(\delta\tau(z))\neq\emptyset$ for $|z|+\delta\tau(z)>r$. Thus by (\ref{averbere}) and (\ref{calcul}), the first integral is dominated by
\begin{align*}
\int_{\D}|R_nf|^2\omega\,d\mu^\phi_{\omega,u,r}\lesssim\sup_{z\in\D}\widehat{\mu^\phi_{\omega,u,r}}_\delta(z)\|R_nf\|_{2,\omega}^2&\leq\sup_{|z|+\delta\tau(z)>r}\widehat{\mu^\phi_{\omega,u}}_\delta(z)\\&\lesssim\sup_{|z|+\delta\tau(z)>r}\widetilde{\mu^\phi_{\omega,u}}(z),
\end{align*}
for all $n>0$. Here, letting $r\rightarrow1$, then $|z|\rightarrow1^-$ since $\tau(z)\rightarrow0$. Thus we obtain the following upper estimate by (\ref{equiessennorm}) and the above inequality,
\begin{align*}
\|uC_\phi\|^2_e=\lim_{n\rightarrow\infty}\sup_{\|f\|_{2,\omega}\leq1}\|(uC_\phi R_n)f\|_{2,\omega}^2&\lesssim\limsup_{|z|\rightarrow1}\widetilde{\mu^\phi_{\omega,u}}(z)\\&=\limsup_{|z|\rightarrow1}B_\phi(|u|^2)(z).
\end{align*}
For the lower estimate, for any compact operator $K$ on $A^2(\omega)$, we have
\begin{align*}
\|uC_\phi\|^2_e\geq\|uC_\phi-K\|^2\geq\limsup_{|z|\rightarrow1}\|(uC_\phi)k_z\|_{2,\omega}^2=\limsup_{|z|\rightarrow1}B_\phi(|u|^2)(z).
\end{align*}
\end{proof}

In addition to Theorem \ref{essentialnorm}, we have the following useful necessary condition for the compactness of $uC_\phi$ on $A^2(\omega)$.

\begin{cor}
Let $\omega\in\mathcal{W}$ and $u$ be an analytic function on $\D$. If the weighted composition operator $uC_\phi$ is compact on $A^2(\omega)$ then
\begin{align}\label{compactcondi}
\lim_{|z|\rightarrow1^-}\frac{\tau(z)^2}{\tau(\phi(z))^2}\frac{\omega(z)}{\omega(\phi(z))}|u(z)|^2=0.
\end{align}
\end{cor}
\begin{proof}
By Lemma \ref{testft}, we know that the normalized kernel sequence $\{k_{z}\}$ converges to $0$ weakly when $|z|\rightarrow1$. Since $(uC_\phi)^*K_z=\overline{u(z)}K_{\phi(z)}$, we have
\begin{align*}
0=\lim_{|z|\rightarrow1}\|(uC_\phi)^*k_{z}\|_{2,\omega}^2&=\lim_{|z|\rightarrow1}|u(z)|^2\frac{K(\phi(z),\phi(z))}{K(z,z)}\\
&\gtrsim \lim_{|z|\rightarrow1}|u(z)|^2\frac{\tau(z)^2}{\tau(\phi(z))^2}\frac{\omega(z)}{\omega(\phi(z))}.
\end{align*}
\end{proof}

\section{Schatten class weighted composition operators}
In order to study the Schatten class composition operator, we use some known results of Toeplitz operators acting on $A^2(\omega)$. First of all, we recall that the definition and some facts of Toeplitz operators defined on $A^2(\omega)$. For a finite complex Borel measure $\mu$ on $\D$, the Toeplitz operator $T_\mu$ is defined by
\begin{align*}
T_\mu f(z)=\int_\D f(\xi)K(z,\xi)\omega(\xi)d\mu(\xi),
\end{align*}
for $f\in A^2(\omega)$.
 Here, it is not clear the integrals above will converge, so we give the following additional condition on $\mu$,
\begin{align}\label{addconi}
\int_\D |K(z,\xi)|^2\omega(\xi) d|\mu|(\xi)<\infty,
\end{align}
for all $z\in\D$. The following lemma is from Lemma 2.2 in \cite{APP}.
\begin{lem}\label{fubini}
Let $\omega\in\mathcal{W}$, $\mu\geq0$ and $\mu$ be a Carleson measure. Then we have
\begin{align*}
\langle T_\mu f,g\rangle_\omega=\int_\D f(\xi)\overline{g(\xi)}\omega(\xi) d\mu(\xi),\quad f,g\in A^2(\omega),
\end{align*}
where $\langle f,g\rangle_\omega=\int_{\D}f(z)\overline{g(z)}\omega(z)\,dA(z)$.
\end{lem}
It is well known that composition operators are closely related to Toeplitz operators on weighted Bergman spaces. From Lemma \ref{fubini}, we have the relation,
\begin{align}\label{ToeCom}
\langle(uC_\phi)^*(uC_\varphi)f,g\rangle_\omega&=\int_\D f(\phi(z))\overline{g(\phi(z))}|u(z)|^2\omega(z)dA(z)\nonumber\\
&=\int_\D f(z)\overline{g(z)}\omega(z)d\mu^\phi_{\omega,u}(z)=\langle T_{\mu_{\omega,u}^\phi} f,g\rangle_\omega,
\end{align}
where $d\mu_{\omega,u}^\phi$ is defined by (\ref{Carlesonmeasure}). Thus, we can use the results of $T_{\mu_{\omega,u}^\phi}$ to see when the composition operators $uC_\varphi$ belong to $S_p$. In other words, it suffices to show when $T_{\mu_{\omega,u}^\phi}$ is in $S_{p/2}$ since $uC_\varphi\in S_p$ is equivalent to $(uC_\varphi)^*(uC_\varphi)\in S_{p/2}$ as we studied in Section 2.4.\\
\indent The following lemma is the characterization of the membership in the Schatten ideals of a Toeplitz operator acting on $A^2(\omega)$.
\begin{lem}[Theorem 4.6 \cite{APP}]\label{app}
Let $\omega\in\mathcal{W}$, $\delta\in(0,m_\tau/4)$ and $0<p<\infty$. For $\mu\geq0$, the following conditions are equivalent:
\begin{itemize}
\item[(1)]$T_\mu\in S_p(A^2(\omega))$.
\item[(2)]$\widetilde{\mu}\in L^p(\Delta \varphi dA)$.
\item[(3)]$\hmu_{\delta}\in L^p(\Delta \varphi dA)$.
\end{itemize}
\end{lem}

\begin{thm}\label{S_pcomposition}
Let $0<p<\infty$. Let $\omega\in\mathcal{W}$ and $u$ be an analytic function on $\D$. Then $uC_\phi\in S_p$ if and only
if $B_\phi(|u|^2)\in L^{p/2}(\Delta\varphi dA)$.
\end{thm}
\begin{proof}
If $uC_\phi\in S_p$ then $(uC_\varphi)^*(uC_\varphi)=T_{\mu_{\omega,u}^\phi}\in S_{p/2}$ by (\ref{ToeCom}). Thus, we have the equivalent condition $\widetilde{\mu^\phi_{\omega,u}}=B_\phi(|u|^2)\in L^{p/2}(\Delta\varphi dA)$ from (\ref{Berezinphi}) and Lemma \ref{app}.
\end{proof}

\begin{cor}\label{schatten}
Let $u$ be an analytic function on $\D$ and let $\phi$ be
an analytic self-map of $\D$. Then $uC_\phi$ is a Hilbert-Schmidt operator if and only if
\begin{align*}
\int_\D \frac{\tau(z)^2}{\tau(\phi(z))^2}\frac{\omega(z)}{\omega(\phi(z))}|u(z)|^2\Delta\varphi(z)dA(z)<\infty.
\end{align*}
\end{cor}
\begin{proof}
By Theorem \ref{S_pcomposition} and (\ref{diagonalesti}), we have
\begin{align*}
&\int_\D|B_\phi(|u|^2)(z)|\Delta\varphi(z)dA(z)\\&=\int_\D\int_\D|k_z(\phi(\xi))|^2|u(\xi)|^2\omega(\xi)\,dA(\xi)\frac{1}{\tau(z)^2}dA(z)\\
&\approx\int_\D\int_\D|K(z,\phi(\xi))|^2|u(\xi)|^2\omega(\xi)\,dA(\xi)\omega(z)dA(z)\\
&=\int_\D K(\phi(\xi),\phi(\xi))|u(\xi)|^2\omega(\xi)\,dA(\xi)\\
&\approx\int_\D \frac{\tau(\xi)^2}{\tau(\phi(\xi))^2}\frac{\omega(\xi)}{\omega(\phi(\xi))}|u(\xi)|^2\Delta\varphi(\xi)dA(\xi)<\infty.
\end{align*}
\end{proof}

{\bf Acknowledgements.} The author would like to thank the referee for indicating various mistakes and giving helpful comments.

\bibliographystyle{amsplain}

\end{document}